\documentclass{amsart} \usepackage{amsmath}
\usepackage{amssymb} \usepackage{amsfonts}
\usepackage{amsthm} \usepackage{enumerate} \usepackage[all]{xy}

 \newcommand{\mb}{\mathbb}
 \newcommand{\ol}{\overline}
\newcommand{\ra}{\rightarrow} 
 
\newcommand{\iso}{\approx} 
\newcommand{\Gal}{\text{Gal}} 
 \newcommand{\tr}{\text{tr}}
\newcommand{\Z}{\mb{Z}} 
\newcommand{\F}{\mb{F}} 
\newcommand{\Q}{\mb{Q}} 
 
 \newcommand{\bs}{\backslash}
\newcommand{\<}{\langle} \renewcommand{\>}{\rangle}
 \newcommand{\vare}{\varepsilon}
 \newcommand{\Cl}{\text{Cl}}

 \renewcommand{\l}{\ell}

\begin{document}
\newtheorem{Theo}{Theorem}[section]
\newtheorem*{Ques}{Question}
\newtheorem{Prop}[Theo]{Proposition}

\newtheorem*{Lemma}{Lemma}
\newtheorem{Cor}[Theo]{Corollary}
\newtheorem{Conj}[Theo]{Conjecture}
\setcounter{tocdepth}{1}
\theoremstyle{definition}
\newtheorem{Exam}[Theo]{Example}
\newtheorem*{Remark}{Remark}
\newtheorem*{RemDef}{Remark/Definition}
\newtheorem{Defn}[Theo]{Definition}
\title{$p$-Tower Groups over Quadratic Imaginary Number Fields}
\author{Cam McLeman}
\thanks{Published as \emph{$p$-Tower Groups over Quadratic Imaginary Number Fields.}  Ann. Sci. Math. Québec  32  (2008),  no. 2, 199--209.}
\address{Willamette University, 900 State Street, Collins Science Center.  Salem, OR.  97304}
\email{cmcleman@willamette.edu}

\begin{abstract}
The modern theory of class field towers has its origins in the study
of the $p$-class field tower over a quadratic imaginary number field,
so it is fitting that this problem be the first in the discipline to
be nearing a solution.  We survey the state of the subject and present
a new cohomological condition for a quadratic imaginary number field
to have an infinite $p$-class field tower (for $p$ odd).  Under an
additional hypothesis, we refine this to a necessary and sufficient
condition and describe an algorithm for evaluating this condition for
a given quadratic imaginary number field.
\end{abstract}
\maketitle 
\section{Introduction}
Let $K$ be number field and $p$ a prime.  Set $K_p^{(0)}:=K$ and for
each $i\geq 1$, let $K_p^{(i)}$ be the Hilbert $p$-class field of
$K_p^{(i-1)}$, i.e., the maximal unramified abelian $p$-extension of
$K_p^{(i-1)}$.  Repeating this construction gives the \emph{Hilbert
$p$-class field tower} (or just \emph{$p$-tower}) over $K$:
\begin{align*}
K=K_p^{(0)}\subset K_p^{(1)}\subset\cdots \subset
K_p^{(i)}\subset\cdots.
\end{align*}
By class field theory, the Galois group of consecutive terms in the
tower is isomorphic to the $p$-part of the class group of the base
field: $\Gal(K_p^{(i)}/K_p^{(i-1)}) \cong \Cl_p(K_p^{(i-1)})$.  The
tower thus stabilizes at the $i$-th step, i.e.,
$K_p^{(i)}=K_p^{(i+1)}=\cdots$, if and only if $\Cl_p(K_p^{(i)})=1$.
If such an $i$ exists, the minimal such one is called the
\emph{$p$-length} $\l_p(K)$ of $K$, and we set $\l_p(K)=\infty$
otherwise.  Let $K_p^{(\infty)}$ be the compositum of all fields in
the $p$-class field tower.  We call $\Gal(K_p^{(\infty)}/K)$ the
\emph{$p$-tower group over $K$}.

\vspace*{.1in}

One arithmetic motivation of class field towers relates to the
question of whether a given number field admits an extension with
certain restrictions on its class number.  For example, the following
well-known lemma addresses the question of whether or not a number
field can be embedded in an extension with class number relatively
prime to a given prime $p$, a subtle condition which famously played a
role in Kummer's proof of the first case of Fermat's Last Theorem for
regular primes.

\begin{Lemma}
Let $K$ be an algebraic number field and $p$ a rational prime.  Then
there is a finite extension of $K$ with class number prime to $p$ if
and only if $\l_p(K)<\infty$.
\end{Lemma}
\begin{proof}
The first direction is clear: If $\l_p(K)=n<\infty$, then
$K_p^{(n)}=K_p^{(n+1)}$ and so $\Gal(K_p^{(n+1)}/K_p^{(n)})\cong
\Cl_p(K_p^{(n)})$ is trivial.  Thus $K_p^{(n)}$ is a finite extension
of $K$ with class number prime to $p$.  For the converse, we use that
the Hilbert $p$-class field is functorial with respect to inclusions,
i.e., that $K\subset L$ implies $K_p^{(1)}\subset L_p^{(1)}$.  By
induction, this gives $K_p^{(\infty)}\subset L_p^{(\infty)}$.  Now
suppose that $L$ is a finite extension of $K$ with class number prime
to $p$.  Then the $p$-tower of $L$ has length 0, and thus $K\subset L$
implies that $K_p^{(\infty)}\subset L_p^{(\infty)}=L_p^{(0)}=L$.  Thus
the entire $p$-tower over $K$ is contained in the finite extension
$L$, so is necessarily finite.
\end{proof}

In general, the theory of $p$-class field towers has an abundance of
problems and only a select few answers, but the case in which $K$ is a
quadratic imaginary number field has historically been one of the few
exceptions to this trend.  To list a few relevant facts (fields
denoted by $K$ below are assumed to be quadratic imaginary):
\begin{itemize}
\item The work of Golod and Shafarevich \cite{GS64} provided the first
  quadratic imaginary number field, $\Q(\sqrt{-4849845})$, with an
  infinite 2-tower, and some further refinements by Koch and Venkov
  \cite{KV74} and Schoof \cite{Sc86} allowed examples to be found for
  other primes (e.g., $\Q(\sqrt{-3321607})$ for $p=3$,
  $\Q(\sqrt{-222637549223})$ for $p=5$).
\item More recently, Bush \cite{Bu03} gave the first examples with
$2$-tower of length 3 (e.g., $K=\Q(\sqrt{-445})$).  For $p>2$, the
longest known finite towers are of length 2 (e.g., \cite{ST34}
determined that $\l_3(\Q(\sqrt{-3299}))=2$).
\item As we shall see, of central importance is the $p$-rank of the
class group of $K$.  One can find fields with large $p$-rank for $p=2$
by the classical genus theory of Gauss, and for other primes by the
work of Yamamoto \cite{Ya70}, Mestre \cite{Me83}, and Buell
\cite{Bu76}, among others.
\end{itemize}

\vspace*{.1in}

As evidenced by the above list, there is a noticeable divide in the
theory depending on whether $p$ is even or odd.  For $p=2$, the theory
benefits from the use of Kummer theory and from the observation that
the 2-class field tower $K_2^{(\infty)}$ over a quadratic imaginary
number field $K$ is a subfield of $\Q_S(2)$, the maximal 2-extension
of $\Q$ unramified outside the set $S$ of primes ramifying in $K/\Q$.
As we will discuss, the Galois groups $G_S(p):=\Gal(\Q_S(p)/\Q)$ are
in general more forthcoming with information concerning their relation
structure, and this descends to give refined information about the
corresponding 2-tower groups.  This paper will focus instead on the
case in which $p$ is odd, which while missing out on the above
benefits, capitalizes instead on slightly stronger versions of the
Golod-Shafarevich theorem and a calculation of Shafarevich, both of
which we describe in the next section.

\section{Generators and Relations for $p$-Tower Groups}

Our intuition tells us (and it is not difficult to prove) that if one
is to present a group with generators and relations, the number of
generators is certainly a lower bound for the number of relations
needed to keep the group finite.  A famous theorem of Golod and
Shafarevich, which we will discuss shortly, makes this precise and
improves this bound significantly.  Similarly, one intuitively
recognizes that a long or complicated relation contributes less to
keeping a group finite than does a short one.  As a rough illustration
of this idea, we note that the group given by the presentation
$\<x\,|\,x^{n}\>$ increases in size as $n\ra\infty$.  Unfortunately,
most naive attempts at making this rigorous fail.  As an illustration
of the potential difficulties, observe that the presentation
$\<x,y\,|\,\,y^{-1}xyx^{-2},\,x^{-1}yxy^{-1}\>$ defines the infinite
cyclic group, whereas there is a well-known presentation
$\<x,y\,|\,y^{-1}xyx^{-2},\,x^{-1}yxy^{-2}\>$ defining the trivial
group, despite having a longer (and only marginally different) second
relation.  A more complete version of the Golod-Shafarevich theorem
mentioned above makes rigorous this idea, giving a refinement which
takes into account a measure of how ``complicated'' the relations are.
More precisely, we will assign an invariant to each defining relation
corresponding to its depth with respect to a certain filtration on the
group.

\begin{Defn}
Let $G$ be a group and let $\F_p[G]$ be its group ring over $\F_p$.
The degree map $\vare:\F_p[G]\ra \F_p$, given by $\sum a_ig_i\ra \sum
a_i$, is a surjective homomorphism whose kernel $I$ is called the
\emph{augmentation ideal} of $\F_p[G]$ (or just of $G$).  Define the
\emph{$n$-th (modular) dimension subgroup $G_n$ of $G$} (the $p$ is
suppressed in the notation, but it will always be clear which prime we
are using) by
\begin{align*}
G_n=\{g\in G\mid g-1\in I^n(G)\}.
\end{align*}
Note that we will always write $G$ multiplicatively, so the additive
notation $g-1$ will always refer to addition in the group ring
$\F_p[G]$.  We have $G_1=G$ trivially, and the aforementioned
filtration is the descending chain of subgroups
\begin{align*}
G=G_1\supseteq G_2\supseteq G_3\supseteq\cdots.
\end{align*}
We call this filtration the \emph{Zassenhaus filtration} of $G$.
\end{Defn}
\begin{Remark}
Both names -- the ``Zassenhaus filtration'' and the ``modular
dimension subgroups'' (in addition to several others) -- appear
frequently in the literature, apparently as historical remnants from
before it was known that various definitions all gave the same series
of subgroups.  We will use both terms as they seem to fit most
naturally in the English language -- the filtration will refer to the
sequence of subgroups, whereas the dimension subgroups will refer to
the specific terms in the chain.
\end{Remark}

Our principal tool for studying this filtration will be the following
theorem of Lazard which relates the Zassenhaus filtration to the lower
central series (defined recursively in terms of commutators by
$\gamma_1(G):=G$, $\gamma_n(G):=[\gamma_{n-1}(G),G]$).

\begin{Theo}[Lazard, \cite{La54}]\label{lazard}
For any group $G$ and any prime $p$, the $n$-th dimension subgroup
$G_n$ of $G$ is given by
\begin{align*}
G_n=\prod_{ip^j\geq n}\gamma_i(G)^{p^j}.
\end{align*}
\end{Theo}

Note that while the product is over infinitely many pairs $(i,j)$, all
but finitely many of these are redundant since we have the inclusions
$\gamma_i(G)\subset \gamma_j(G)$ for $i\geq j$ and $\gamma_i(G)^{p^j}\subset
\gamma_i(G)^{p^k}$ for $j\geq k$.  For example, one can easily argue
from the theorem that $G_1=G$ and that $G_2=G^p[G,G]$.

\vspace*{.1in}

For a pro-$p$-group $G$, the quantities $d(G):=\dim_{\F_p}H^1(G,\F_p)$
and $r(G):=\dim_{\F_p}H^2(G,\F_p)$ are the respective cardinalities of
a minimal generating set and set of relations for $G$ as a
pro-$p$-group, i.e., there exists a minimal presentation $1\ra R\ra
F\ra G\ra 1$ where $F$ is a free pro-$p$-group on $d$ generators and
$R$ is $r$-generated as a normal subgroup of $F$.  We note for later
the crucial fact that the Burnside Basis Lemma for pro-$p$-groups
implies that for a $p$-tower group, $d(G)=d(G/[G,G])=d_p\Cl(K)$, where
$d_p$ denotes the $p$-rank of an abelian group.  For each $r\in R$,
define the \emph{level} of $r$ to be the greatest integer $k$ such
that $r\in F_k$, the $k$-th dimension subgroup of $F$.  Given a set of
generators $\{\rho_i\}_{i=1}^r$ of $R$, let $r_k$ denote the number of
$\rho_i$ which have level $k$ (so $\sum_{k=1}^\infty
r_k=r$)\footnote{We caution the reader that the use of the symbols
$r_1$ and $r_2$ in relation to the number of real and complex embeddings
of a number field will not be needed for this paper, so these
quantities will always refer to the number of relations of levels 1
and 2 respectively.}.  The following is a small but important step in
determining the relation structure of a pro-$p$-group.

\begin{Prop}\label{rone}
For any minimal presentation of a $d$-generated pro-$p$-group $G$, we
have $r_1=0$.
\end{Prop}
\begin{proof}
Write $G\cong F/R$ with $F$ the free pro-$p$-group on $d$ generators,
and $R=\<\rho_1,\ldots,\rho_r\>$.  By Theorem \ref{lazard}, we have
$F_2=F^p[F,F]$, which coincides with the Frattini subgroup of
non-generators of $F$.  If there were a relation $\rho$ of level 1, it
would lie in $F\bs F^p[F,F]$ and hence would be non-trivial in
$F/F^p[F,F]$, so could be used to reduce the number of generators,
contradicting that $G$ was $d$-generated.
\end{proof}

\noindent
The theorem of Golod-Shafarevich can now be made precise.  

\begin{Theo}[Golod-Shafarevich]\label{fullgs}
Let $G$ be a pro-$p$-group with $d$ generators, and choose a
presentation $G\iso F/R$ of $G$.  Choose a generating set
$\{\rho_i\}_{i\in I}$ for $R$ as a normal subgroup of $F$, and let
$r_k$ denote the number of these relations which have level $k$.  Then
if $G$ is finite, we have that
\begin{align*}
\sum_{k=2}^\infty r_kt^k-dt+1>0
\end{align*}
for all $t\in(0,1)$.  
\end{Theo}

Given a group abstractly (e.g., a Galois group), it may in practice be
difficult to find a presentation and hence apply the theorem.
However, the versatility of the theorem means that partial results and
even trivial bounds can be put to good use.  For example, to obtain
the most frequently stated (but a relatively weak) form of the
Golod-Shafarevich theorem, we need only observe that for any
presentation we have $\sum r_k=r$ and that $t^k\leq t^2$ for all
$t\in[0,1]$, so that the theorem gives $rt^2-dt+1>0$, which as in the
proof of the following corollary, quickly implies that
$r>\frac{d^2}{4}$ for finite $p$-groups. The natural generalization of
this idea makes for a ``medium strength'' Golod-Shafarevich theorem:

\begin{Cor}[\cite{Ko69}]\label{medgs}
With notation as in the theorem, suppose further that $R\subset F_m$.
Then if $G$ is finite, we have
\begin{align*}
r>d^m\frac{(m-1)^{m-1}}{m^m}.
\end{align*}
\end{Cor}
\begin{proof}
Since $R\subset F_m$, we have $r_k=0$ for $k<m$, and so the theorem gives
\begin{align*}
\sum_{k=m}^\infty r_kt^k-dt+1>0
\end{align*}
for $t\in(0,1)$.  Since $\sum r_k=r$ and $t^k\leq t^m$ on this
interval for any $k\geq m$, the Golod-Shafarevich inequality gives us
that $rt^m-dt+1>0$.  Plugging in
$t=\left(\frac{d}{mr}\right)^{1/(m-1)}\in(0,1)$ gives a contradiction if
$r\leq d^m\frac{(m-1)^{m-1}}{m^m}$.
\end{proof}

\vspace*{.1in} As is clear from the previous argument, there is much
flexibility in the application of this inequality.  Namely, if the
inequality holds for some set of relation levels, then it holds again
if we move (e.g., via Tietze transformations) one of the relations
into a higher (less deep) level.  It is therefore prudent to determine
the presentation with deepest possible levels, on which the inequality
will deliver the strongest possible result.  To this end, we follow an
inductive procedure to construct a presentation with relations of
maximal depth.  Namely, let $R_1=\emptyset$, and define recursively
\begin{align*}
R_k=R_{k-1}\cup\{\rho_{k,1},\ldots,\rho_{k,r_k}\},
\end{align*}
where the $\rho_{n,i}$ are chosen to be a minimal generating set for
$RF_{n+1}/F_{n+1}$ as a normal subgroup of $F/F_{n+1}$.  Whereas
previously the values of $r_k$ depended on the choice of presentation,
we redefine the $r_k$ to be the quantities arising in this process so
that they become invariants of the group (as opposed to invariants of
the presentation).  Fixing the values of $r_k$ correspondingly fixes a
new invariant of $G$, the \emph{Zassenhaus polynomial} $Z_G(t):=\sum
r_kt^k-dt+1$ of $G$ appearing on the left-hand side of the
Golod-Shafarevich inequality.

\vspace*{.1in}

The various forms of the Golod-Shafarevich theorem are of particular
number-theoretic importance when applied to Galois groups for which
one knows more about the quantities $d$ and $r$.  For $p$-tower
groups, this information is provided by the following remarkable
calculation of Shafarevich (which further illustrates the divide
between $p=2$ and $p$ odd mentioned in the introduction).

\begin{Theo}[Shafarevich, \cite{Sh63}]\label{r=d}
For $G$ a $p$-tower group over a quadratic imaginary number field, we
have $r(G)-d(G)\leq 1$.  For $p\neq 2$, this is strengthened to
$r(G)=d(G)$.
\end{Theo}

One immediately forces a contradiction by contrasting this theorem
with Corollary \ref{medgs}: Since $R\subset F_2$ by Proposition
\ref{rone}, Corollary \ref{medgs} implies we must have
$r>\frac{d^2}{4}$ for $G$ to be finite.  Combining this with either of
the two inequalities in Theorem \ref{r=d} (depending on the prime)
easily contradicts a finiteness assumption on $G$ for quadratic
imaginary number fields $K$ with sufficiently large $d$.  Such $K$ can
frequently be found in practice.  For example, armed with Gauss' genus
theory it is easy to construct $K$ with arbitrarily large
$d=d_2\Cl(K)$, leading to Shafarevich's first examples of quadratic
imaginary number fields with infinite $2$-class field towers.

\vspace*{.1in}

Making further progress restricting the class of $p$-groups which can
arise as $p$-tower groups requires more information on the form of the
relations defining such groups.  As a successful example of such an
enterprise, we note that a more general form of Shafarevich's
calculation (see \cite{NSW00}, VIII.7) computes the generator and
relation ranks for $G_S(p)$ as well.  Specifically, if one insists
that $p\notin S$ and that $S$ be minimal in the sense that it contains
no primes that cannot ramify in a $p$-extension, then the computation
gives $d(G_S(p))=r(G_S(p))=|S|$.  In this case, one can use the
relations defining the Galois group of the maximal $p$-extension of
$\Q_\l$ for each $\l\in S$ (the so-called ``local relations'') to
extract information about the relation structure of $G_S(p)$.
Relationships between the primes in $S$ -- for example, the $p$-th
power residue symbols $\left(\frac{\l_i}{\l_j}\right)_p$ for
$\l_i,\l_j\in S$ -- influence the form of the local relations, which
in turn can be used in conjunction with the Golod-Shafarevich theorem
to prove $G_S(p)$ infinite for certain $S$ (for example, see
\cite{Mo02}).  In contrast, the $p$-tower groups for $p$ odd also
satisfy $r=d$, but the relations are of a much more mysterious nature
(so-called ``unknown relations'').  Nonetheless, we do have the
following result of Koch and Venkov, providing a significant
refinement of the admissible relation structure of $p$-tower groups.

\begin{Theo}[Koch-Venkov, \cite{KV74}]\label{KV}
Let $G$ be a $p$-tower group over a quadratic imaginary number field.
Then $r_{2k}=0$ for all $k\geq 1$.  In particular, since $r_1=r_2=0$,
we have $R\subset F_3$.
\end{Theo}
\begin{Remark}
The original proof of this fact is fairly straight-forward linear
algebra, showing that generators and relations for the group can be
chosen to be in the $-1$ eigenspace of the action $\sigma:G\ra G$
induced by complex conjugation (i.e., so that complex conjugation
sends generators and relations to their inverses).  It then follows
that if $r\in F_k$, then $\ol{r}\in F_k/F_{k+1}$ satisfies both
$\ol{r}^\sigma=(-1)^k \ol{r}$ and $\ol{r}^\sigma=-\ol{r}$, implying
that $\ol{r}=e$ for $k$ even and thus that $r\in F_{k+1}$.  It is
worth mentioning that a more modern version of this argument is given
in \cite{KL89} using the triviality of the cup product
$H^1(G,\F_p)\times H^1(G,\F_p)\ra H^2(G,\F_p)$, but only proves
$r_2=0$.  It is unclear if the more modern cohomological framework can
provide the full result.
\end{Remark}
\noindent
As promised above, this extra information about the levels of the
defining relations strengthens the results on conditions for
guaranteeing infinite $p$-class field towers.

\begin{Cor}[Koch-Venkov]
Let $p\neq 2$ be prime, and let $K$ be a quadratic imaginary number
field with $d_p\Cl(K)\geq 3$.  Then $K_p^{(\infty)}/K$ is infinite.
\end{Cor}
\begin{proof}
By Theorem \ref{KV}, we have $r_2=0$, so $R\subset F_3$.  Suppose
$G:=\Gal(K_p^{(\infty)}/K)$ were finite.  Then by Theorem \ref{medgs},
we have
\begin{align*}
r>\frac{4d^3}{27}.
\end{align*}
Recalling that $r=d$ for $p$-tower groups ($p\neq 2$), this gives a
contradiction for $d\geq 3$. 
\end{proof}
\begin{Remark}
The analogous result for $p=2$ is that one needs the 4-rank of $\Cl(K)$
to exceed 2 (i.e., that $\Cl(K)$ contains a subgroup of type
$(4,4,4)$), and it is conjectured that a 2-rank exceeding 3 also
suffices. See \cite{Ha96}.
\end{Remark}

This result nearly completes the analysis of whether or not a given
quadratic imaginary field has an infinite $p$-class field tower, since
we have that such a field's $p$-tower length is given by
\begin{align*}
\l_p(K)=\begin{cases}
0&\text{ if }d_p\Cl(K)=0\\
1&\text{ if }d_p\Cl(K)=1\\
?&\text{ if }d_p\Cl(K)= 2\\
\infty&\text{ if }d_p\Cl(K)\geq 3.\\
       \end{cases}
\end{align*}
\begin{proof}
The rank $\geq 3$ case was just proven.  The rank $0$ case is the case
that $\Cl_p(K)$ is trivial, and hence $K$ is its own Hilbert $p$-class
field.  The rank $1$ case follows from an application of the Burnside
Basis Lemma that $d(G)=d(G/[G,G])$ for pro-$p$-groups.  Namely, if $G$
is a $p$-tower group such that $G/[G,G]\iso \Cl_p(K)$ is 1-generated,
then $G$ itself is 1-generated.  Thus $G$ is cyclic, and hence
abelian, and thus isomorphic to $\Cl_p(K)$ by maximality of the
Hilbert class field.  Thus $K_p^{(1)}=K_p^{(\infty)}$ by Galois
theory, and so the tower has length 1.
\end{proof}

Thus the only case remaining undecided occurs when $d_p\Cl(K)=2$, and
even the full version of the Golod-Shafarevich theorem above does not
rule out these cases.  Indeed, there are many quadratic imaginary
number fields with finite $3$-towers and $d=2$, the earliest
discovered being the example of Scholz and Taussky mentioned in the
introduction.  It is noteworthy that no longer 3-towers (or $p$-towers
for \emph{any} odd $p$) have been found since.  In fact, to the
author's knowledge, there are no known examples, for $p$ odd, of
quadratic imaginary number fields with $d_p\Cl(K)=2$ and an infinite
$p$-class field tower.  Nonetheless, Theorem \ref{fullgs} gives
insight in to this case as well.  Specifically, since in this case we
have $r=d=2$, the Zassenhaus polynomial of a finite such group is of the form 
\begin{align*}
Z_G(t)=t^i+t^j-2t+1,
\end{align*}
where $i$ and $j$ are odd (by Theorem \ref{KV}) and at least three
(by Proposition \ref{rone}).  One checks that the polynomials
\begin{align*}
t^3+t^a-2t+1\quad\quad\text{ and }\quad\quad t^5+t^b-2t+1
\end{align*}
\noindent
both have zeroes in the interval $(0,1)$ for values of $a\geq 9$ or
$b\geq 5$.  Thus every finite 2-generated $p$-tower group admits a
presentation with one of the following three Zassenhaus polynomials:
\begin{align*}
Z_G(t)\in\{t^7+t^3-2t+1,t^5+t^3-2t+1,t^3+t^3-2t+1\},
\end{align*}
\noindent i.e., the two relations defining a finite $p$-tower group
lie in levels $i$ and $j$, where $(i,j)\in\{(3,3),(3,5),(3,7)\}$.  We
call this pair $(i,j)$ the \emph{Zassenhaus type} of $G$, which proves
to be an important invariant in the characterization of $p$-tower
groups.  This represents the state of the art in the sense that all
three of the possible Zassenhaus types are still \emph{a priori}
possible, though two points merit mentioning:
\begin{itemize}
\item Every known example of an explicit 2-generated $p$-tower group
is of Zassenhaus type $(3,3)$, including the groups investigated in
\cite{ST34} and the series of groups $G_n$ described in \cite{BB06}.
\item For $p>7$, the author has shown in \cite{Mc09} that if $G$ is a
$p$-tower group of Zassenhaus type $(3,7)$ over a quadratic imaginary
number field $K$ with $\Cl_p(K)\iso (p^a,p^b)$ ($1\leq a\leq b$), then
$|G|\geq p^{21+a+b}\geq p^{23}$.  There are no known groups which
satisfy this collection of properties.
\end{itemize}

\noindent
This leads to the following (somewhat unsubstantiated) conjecture.

\begin{Conj}[The (3,3) Conjecture]
A $p$-tower group over a quadratic imaginary field is finite if and
only if it is of Zassenhaus type (3,3).
\end{Conj}

While the two bullets above represent the scant, but slowly building,
body of evidence in support of this conjecture, it is worth mentioning
because of the elegant solution to the $p$-class field tower problem
one obtains as a corollary.  This will be presented in the next
section after we set up some terminology and notation.

\section{A Criterion for Infinite $p$-Towers}

Let $G$ be a $d$-generated pro-$p$-group, and take a minimal
presentation $1\ra R\ra F\ra G\ra 1$ for $G$.  The corresponding
inflation-restriction sequence in Galois cohomology gives isomorphisms
$\text{inf}\colon H^1(G,\F_p)\ra H^1(F,\F_p)$ and $\text{tg}\colon
H^1(R,\F_p)^G\ra H^2(G,\F_p)$, where $tg$ denotes the transgression
map.  For $\rho\in R$, we define the corresponding \emph{trace map}
$\tr_\rho:H^2(G,\F_p)\ra \F_p$ by
$\tr_\rho(\phi):=(\text{tg}^{-1}\phi)(\rho)$.  Let
$\chi_1,\ldots,\chi_d$ be a basis for $H^1(G,\F_p)$.  A frequently
used tool from Galois cohomology (\cite{NSW00}, 3.9.13) is that a
relation $\rho\in R\subset F$ defining $G$ can be computed modulo
$F_3$ from the scalars $\tr_\rho(\chi_i\cup\chi_j)$ (here, $\cup$
denotes the cup product).

\vspace*{.1in}

If $K$ is a quadratic imaginary number field with $d_p(\Cl(K))=2$ and
$G$ is the $p$-tower group over $K$, then as mentioned in the remark
after Theorem \ref{KV}, the cup product is trivial, rendering the
above technique rather unhelpful.  However, the vanishing of the cup
product is precisely the required condition for the so-called triple
Massey products
\begin{align*}
\<\cdot,\cdot,\cdot\>\colon H^1(G,\F_p)\times H^1(G,\F_p)\times H^1(G,\F_p)\ra H^2(G,\F_p)
\end{align*}
to be well-defined.  To define these, let $\{\chi_1,\chi_2\}$ be a
basis for $H^1(G,\F_p)$.  For each pair $(i,j)$ with $1\leq i,j\leq
2$, the triviality of the cup-product implies we can write
$\chi_i\cup\chi_j=d(f_{ij})$ for some 1-chain $f_{ij}:G\ra \F_p$.
Then for any $1\leq i,j,k\leq 2$, we can uniquely define the
\emph{triple Massey products} $\<\chi_i,\chi_j,\chi_k\>$ to be the
class in $H^2(G,\F_p)$ of the 2-cocycle $[\chi_i\cup f_{jk}+f_{ij}\cup
\chi_k]$.  In general, these are well-defined only up to the
collection of choices of $f_{ij}$ but the vanishing of the cup-product
for $p$-tower groups removes any such ambiguity.

\begin{Theo}\label{massey}
Let $K$ be a quadratic imaginary number field with $d_p\Cl(K)=2$, let
$G=\Gal(K_p^{(\infty)}/K)$, choose a basis $\{\chi_1,\chi_2\}$ for
$H^1(G,\F_p)$, and suppose $p>3$.  Then $\l_p(K)=\infty$ if the triple
Massey products $\<\chi_1,\chi_2,\chi_1\>$ and
$\<\chi_1,\chi_2,\chi_2\>$ both vanish.  For $p=3$, we need in
addition the triviality of the triple Massey products
$\<\chi_1,\chi_1,\chi_1\>$ and $\<\chi_2,\chi_2,\chi_2\>$.
\end{Theo}
\begin{proof}
By Theorem \ref{r=d} and the computations that $r_1=r_2=0$, the
hypotheses of the theorem guarantee the existence of a presentation
$1\ra R\ra F\ra G\ra 1$, where $F=\<x,y\>$ is the free pro-$p$-group
on two generators, and $R$ can be generated as a normal subgroup of
$F$ by two elements $\rho_1$ and $\rho_2$ of level at least 3.  In
\cite{Mc08}, the \emph{dimension factors} $\dim_{\F_p}F_n/F_{n+1}$ of a
free pro-$p$-group on $d$ generators are computed, and we find that
\begin{align*}
\dim_{\F_p}F_3/F_{4}=\begin{cases}
4&\text{ if }\,p=3\\
2&\text{ if }\,p\neq 3	       \end{cases}
\end{align*}
for $d=2$.  For $p\neq 3$, a basis for $F_3/F_4$ is given by the
triple commutators $[x,y,x]:=[[x,y],x]$ and $[x,y,y]:=[[x,y],y]$.  We
can thus write $\rho_i=[x,y,x]^{a_i}[x,y,y]^{b_i}\rho_i'$,
$i\in\{1,2\}$ for some $a_i,b_i\in \F_p$, $\rho'_i\in F_{4}$.  By
work of Vogel \cite{Vo04} extending the aforementioned link between
cup products and relation structures, we have
$2a_i=tr_{\rho_i}\<\chi_1,\chi_1,\chi_2\>$ and
$b_i=tr_{\rho_i}\<\chi_1,\chi_2,\chi_2\>$, where the $tr_{\rho_i}$ are
the trace maps defined at the start of the section.  In particular, if
both triple Massey products vanish, then $a_1=b_1=a_2=b_2=0$.  But
this implies $\rho_i=\rho_i'\in F_{4}$ for $i=1,2$, and by the
condition that $r_{2k}=0$, we must further have $\rho_i\in F_{5}$ for
$i=1,2$.  Let $j_1$ and $j_2$ denote the levels of the two relations,
so that $j_1,j_2\geq 5$ by the previous sentence.  If $G$ were finite,
then by Theorem \ref{fullgs}, we would have
\begin{align*}
0<Z_G(t)=t^{j_1}+t^{j_2}-2t+1\leq 2t^5-2t+1
\end{align*}
on the unit interval, which gives a contradiction when evaluated at
$t=\frac{2}{3}$.  The proof is almost identical for $p=3$, noting that
a basis for $F_3/F_4$ is given by $\{x^3,y^3,[x,y,x],[x,y,y]\}$ and
the exponents of $x^3$ and $y^3$ occurring in the representation of
$\rho_i$ modulo $F_4$ are respectively given by
$\tr_{\rho_i}\<\chi_1,\chi_1,\chi_1\>$ and
$\tr_{\rho_2}\<\chi_2,\chi_2,\chi_2\>$.
\end{proof}

Note that it is clear from the proof that all that is really needed is
the vanishing of the traces of the Massey products and not the Massey
products themselves (though, in practice, one could not evaluate these
without knowing more about the relations, leading to a vicious circle
of ignorance).  In short, the proof observes that the vanishing of
these traces of Massey products forces the relations down into the
fourth level of the Zassenhaus filtration, which when combined with
the Koch-Venkov result that $r_4=0$, pushes them down into at least
the fifth level.  One needs then only observe that by the theorem of
Golod and Shafarevich, two relations in the fifth level of the
Zassenhaus filtration is insufficient to keep a 2-generated
pro-$p$-group finite.

\vspace*{.1in}

On a similar note, if one assumes the (3,3) conjecture given in the
last section, then to prove the group infinite it suffices to force
only one of the two relations down to the fourth level, and we are
left with the following particularly elegant cohomological solution to
the $p$-class field tower problem over quadratic imaginary number
fields.

\begin{Prop}[Assuming the $(3,3)$ Conjecture]
Let $K$ be a quadratic imaginary number field with $d_p\Cl(K)=2$, let
$G=\Gal(K_p^{(\infty)}/K)$, choose a basis $\{\chi_1,\chi_2\}$ for
$H^1(G,\F_p)$, and suppose $p>3$.  Then $\l_p(K)<\infty$ if and only
if the matrix
\begin{align*}
\begin{bmatrix}
\tr_{\rho_1}\<\chi_1,\chi_1,\chi_2\>&\tr_{\rho_1}\<\chi_2,\chi_2,\chi_1\>\\
\tr_{\rho_2}\<\chi_1,\chi_1,\chi_2\>&\tr_{\rho_2}\<\chi_2,\chi_2,\chi_1\>
\end{bmatrix}
\end{align*}
is invertible over $\F_p$.
\end{Prop}

Finally, we observe that the $(3,3)$ conjecture would also provide an
algorithm for determining whether or not a given quadratic imaginary
number field $K$ has a finite $p$-class field tower.  Namely, one
first computes the $p$-rank $d_p\Cl(K)$.  If this $p$-rank is 0 or 1,
the $p$-tower is finite, and if the $p$-rank is 3 or larger, the
$p$-tower is infinite.  The only remaining case is that $d=2$.  Note
that $G''\subset \gamma_4(G)\subset G_4$ and so we have a surjection
$\Gal(K_p^{(2)}/K)=G/G''\ra G/G_4$.  Computing the Galois structure of
the first two steps of the $p$-class field tower over $K$ thus allows
the computation of $G/G_4$. In \cite{Mc08}, the author finds that the
quantity $\dim_{\F_p}G_3/G_4$ distinguishes finite $p$-groups of
Zassenhaus type $(3,3)$ from finite $p$-groups of Zassenhaus type
$(3,5)$ or $(3,7)$ for any value of $p$.  Specifically, we have the
following table of values of $\dim_{\F_p}G_3/G_4$ depending on the
prime $p$ and the Zassenhaus type of $G$:
\begin{align*}
\begin{tabular}{c|c|c}
&$(3,3)$&$(3,5)$ or $(3,7)$\\\hline
$p=3$&2&3\\
$p\geq 5$&0&1\\
\end{tabular}
\end{align*}
Thus under the assumption of the (3,3) conjecture, the decision
process concludes with the computation of the quantity
$\dim_{\F_p}G_3/G_4$.  One has $\l_p(K)<\infty$ if the result is even,
and $\l_p(K)=\infty$ if the result is odd.

\vspace*{.1in}

\noindent 
The author would like to thank in particular William McCallum, John
Labute, and Denis Vogel for many useful conversations, the referee for
several improvements, and the organizers and attendees of the
conference for the invitation and feedback.

\end{document}